\documentclass[12pt]{amsart}

 \usepackage[pagewise]{lineno}
 
\usepackage[colorlinks=true,pagebackref,hyperindex,citecolor=blue,linkcolor=red]{hyperref}
\usepackage{amsmath}
\usepackage{amsfonts}
\usepackage{amssymb}
\usepackage{setspace}
\usepackage{moreverb}
\usepackage[latin1]{inputenc}
\usepackage[T1]{fontenc}
\usepackage{amssymb}
\usepackage{tikz}
\usepackage{color}
\usepackage[all]{xy}
\usepackage{xfrac}
\usepackage[top=1in, bottom=1in, left=1in, right=1in]{geometry}
\usepackage{mathrsfs}

\usetikzlibrary{shapes}

\newtheorem{theoremx}{Theorem}

\newtheorem{theorem}{Theorem}[section]

\newtheorem{proposition}[theorem]{Proposition}

\theoremstyle{definition}
\newtheorem{definition}[theorem]{Definition}

\newtheorem{example}[theorem]{Example}

\newtheorem{conjecture}[theorem]{Conjecture}

\newtheorem{remark}[theorem]{Remark}
\numberwithin{equation}{subsection}

\newcommand{\NN}{\mathbb{N}}
\newcommand{\CC}{\mathbb{C}}

\newcommand{\cM}{\mathcal{T}}
\newcommand{\cL}{\mathcal{L}}

\newcommand{\cJ}{\mathcal{J}}
\newcommand{\cT}{\mathcal{T}}

\newcommand{\pT}{\CC[x_1,\ldots,x_s]/\langle f,\cJ_n(f)\rangle}

\newcommand{\Der}{\operatorname{Der}}
\newcommand{\Jac}{\operatorname{Jac}}
\newcommand{\conv}{\CC\{x_1,\ldots,x_s\}}

\newcommand{\pol}{\CC[x_1,\ldots,x_s]}

\newcommand{\dw}{\deg_w}

\begin{document}
\newcommand{\tens}{\otimes}
\newcommand{\hhtest}[1]{\tau ( #1 )}
\renewcommand{\hom}[3]{\operatorname{Hom}_{#1} ( #2, #3 )}

\allowdisplaybreaks

\title[Non-existence of negative derivations]{Non-existence of negative derivations on the higher Nash blowup local algebra}

\author[W. Badilla-C\'espedes]{Wágner Badilla-Céspedes}\address{Escuela de Matemática, Universidad de Costa Rica, San José 11501, Costa Rica.}\email{wagner.badillacespedes@ucr.ac.cr}

\author[A. Castorena]{Abel Castorena}\address{Centro de Ciencias Matemáticas, UNAM, Campus Morelia, Morelia, Michoacán, México.}\email{abel@matmor.unam.mx}

\author[D. Duarte]{Daniel Duarte}\email{adduarte@matmor.unam.mx}

\author[L. Núñez-Betancourt]{Luis Núñez-Betancourt}\address{Centro de Investigación en Matemáticas, Guanajuato, Gto., México.}\email{luisnub@cimat.mx}


\begin{abstract}
Let $f\in\pol$ be a weighted homogeneous polynomial with an isolated singularity, and let $\cT_n(f)$ denote its higher Nash blowup local algebra. We show that $\cT_n(f)$ does not admit negative weighted derivations for $n\geq2$. This answers affirmatively a conjecture of Hussain-Ma-Yau-Zuo.
\end{abstract}

\maketitle

\section*{Introduction}

The  Jacobian matrix of order $n$ is a higher-order version of the classical Jacobian matrix. It was introduced as a tool for computing the higher Nash blowup of a hypersurface \cite{Duarte2017}. Higher-order Jacobian matrices were later rediscovered and further developed by several authors \cite{BJNB2019,BD2020,MR4229623}. Ever since, this matrix has seen a wide variety of applications in singularity theory \cite{Duarte2017,BJNB2019,BD2020,MR4229623,AD2021,Barajas2023,HMYZ2023,LDS2023,LeYasuda2025,FeiYe}.

Let $f\in\conv$, where $\conv$ denotes the ring of convergent power series at the origin. Let $\Jac_n(f)$ be the higher Jacobian matrix of $f$ and $\cJ_n(f)$ the ideal generated by the maximal minors of $\Jac_n(f)$ (see Definition \ref{Jacn}). Hussain-Ma-Yau-Zuo defined the higher Nash blowup local algebra as the quotient $\cM_n(f)=\conv/\langle f,\cJ_n(f) \rangle$ \cite{HMYZ2023}. For $n=1$, this is the classical Tjurina algebra. Moreover, the authors also introduced the  algebra of derivations $\cL_n(f)=\Der_{\CC}(\cM_n(f))$. For $n=1$, this is known as the Yau algebra of $f$. The Tjurina and Yau algebras are fundamental objects in the study of hypersurface singularities \cite{MR0674404,Y1991}.

Three conjectures were proposed regarding $\cT_n(f)$: invariance under contact equivalence, non-existence of negative weighted derivations on $\cT_n(f)$ whenever $f$ is weighted homogeneous, and bounds for the dimension of $\cL_n(f)$ \cite{HMYZ2023}. 

The conjecture regarding the invariance of $\cT_n(f)$ under contact equivalence was proved by L\^{e}-Yasuda \cite{LeYasuda2025}. In this paper we prove the conjecture on the non-existence of negative weighted derivations on $\cT_n(f)$. Let us discuss this conjecture in more detail.
 
It is known that $\pT$ is a graded algebra whenever $f$ is a weighted homogeneous polynomial \cite[Theorem 1.7]{BCCD}. In this case, $\cL_n(f)$ is also graded. There is a general conjecture on algebras of derivations of local Artinian graded algebras, known as Halperin's conjecture, stating the non-existence of negative-weight derivations on such algebras \cite{H1999}. Inspired by Halperin's question, Hussain-Ma-Yau-Zuo conjectured that there are no negative weighted derivations on $\pT$. The goal of this paper is to prove this conjecture.

\newpage

\begin{theoremx}\label{theorem-main}
Let $f\in \CC[x_1,\ldots,x_s]$ be a weighted homogeneous polynomial with respect to $w\in\NN^s$ and degree $d$ that defines an isolated hypersurface singularity. Assume that $d\geq2w_1\geq\cdots\geq 2w_s>0$. Then $\pT$ does not admit negative weighted derivations for all $n\geq2$.
\end{theoremx}

\section*{Acknowledgements}
We thank the referee for the careful revision and the comments that improved the presentation of the paper.

The first author was supported by UNAM Posdoctoral Program (POSDOC). The second author is supported by grant PAPIIT IN101226 ``Teor\'ia de Brill-Noether y aplicaciones'' from UNAM, M\'exico. The third author is supported by SECIHTI project CF-2023-G-33. The fourth author was partially supported by SECIHTI Grants CBF 2023-2024-224 and CF-2023-G-33.

\section{Higher-order Jacobian matrix of a weighted homogeneous polynomial}

\begin{definition}[{\cite{Duarte2017,BJNB2019,BD2020}}] \label{Jacn}
Let $f\in\conv$. Denote
$$\Jac_n(f):=\left(\frac{1}{(\alpha-\beta)!}\frac{\partial^{\alpha-\beta}(f)}{\partial x^{\alpha-\beta}}\right)_{\substack{\beta\in \{\beta\in\NN^s|0\leq|\beta|\leq n-1\} \\ \alpha\in \{\alpha\in\NN^s|1\leq|\alpha|\leq n\}}},$$
where we define $\displaystyle{\frac{1}{(\alpha-\beta)!}\frac{\partial^{\alpha-\beta}(f)}{\partial x^{\alpha-\beta}}=0},$ whenever $\alpha_i<\beta_i$ for some $i$ or if $\beta=\alpha$. It is a $M\times(N-1)$-matrix, where $M=\binom{n+s-1}{s}$ and $N=\binom{n+s}{s}$. We call $\Jac_n(f)$ the Jacobian matrix of order $n$ of $f$. 

Moreover, denote as $\cJ_n(f)$ the ideal generated by all maximal minors of $\Jac_n(f)$. Notice that $\Jac_1(f)$ is the usual Jacobian matrix and $\cJ_1(f)$ is the usual Jacobian ideal of $f$.
\end{definition}

\begin{example}\label{j2}
Let $f\in\CC[x_1,x_2]$. Write $f_i=\frac{\partial f}{\partial x_i}$ and $f_{ij}=\frac{\partial^2 f}{\partial x_ix_j}$. Then 
\[\Jac_2(f)=
\left( 
\begin{array}{ccccc}
             f_1& f_2 & \frac{1}{2}f_{11} & f_{12} & \frac{1}{2}f_{22} \\
             0 & 0 & f_1 & f_2 & 0 \\
             0 & 0 & 0 & f_1 & f_2\\
\end{array} 
\right).\] 
Moreover,
$$\cJ_2(f)=\langle f_1^3,f_1^2f_2,f_1f_2^2,f_2^3,\frac{1}{2}f_1^2f_{22}-f_1f_2f_{12}+\frac{1}{2}f_2^2f_{11}\rangle\subset\CC[x_1,x_2].$$
\end{example}

Recall that a polynomial $f=\sum_{\alpha}c_{\alpha}x^{\alpha}\in\pol$ is called weighted homogeneous of weighted degree $d$ with respect to $w\in\NN_{\geq1}^s$ if $\alpha\cdot w=d$ for all $c_{\alpha}\neq0$. We denote the weighted degree of $f$ by $\dw(f)$. 

\begin{theorem}[{\cite[Theorem 1.7]{BCCD}}]\label{theorem:Tjurina-graded}
Let $f$ be a weighted homogeneous polynomial in $\CC[x_1,\ldots,x_s]$. Then $\cJ_n(f)$ is a weighted homogeneous ideal for every $n \in \NN$.   
\end{theorem}

Hussain-Ma-Yau-Zuo used the ideal $\cJ_n(f)$ to introduce the following generalization of the Tjurina algebra of a hypersurface.

\begin{definition}[{\cite[Definition 1.3]{HMYZ2023}}] \label{Tjurina}
Let $f\in\conv$. Denote
$$\cT_n(f)=\conv/\langle f,\cJ_n(f) \rangle.$$
$\cT_n(f)$ is called the higher Nash blowup local algebra of $f$. Moreover, denote $\cL_n(f)=\Der_{\CC}(\cT_n(f))$, i.e., the Lie algebra of derivations of $\cT_n(f)$. Notice that $\cT_1(f)$ and $\cL_1(f)$ are the Tjurina algebra and the Yau algebra of $f$, respectively.
\end{definition}

\begin{remark}
There are two versions of the higher-order Jacobian matrix in the literature. The original one has $f$ in the $(\beta,\beta)$-entry of $\Jac_n(f)$ \cite{Duarte2017}, whereas the definition given in \cite{BJNB2019,BD2020} has 0 in the $(\beta,\beta)$-entry. Both versions coincide modulo $f$. In particular, the definition of $\cT_n(f)$ is the same regardless of the version that is adopted.
\end{remark}

In view of Theorem \ref{theorem:Tjurina-graded}, the algebra of derivations of $\pT$ is naturally graded. Hussain-Ma-Yau-Zuo proposed the following conjecture, which they verified for $s=n=2$ \cite[Theorem B]{HMYZ2023}. Later, the conjecture was verified for $s=2$ and $n=3$ \cite[Theorem B]{Tjurinathree}. Finally, the case $s=2$ and $n\geq2$ was settled in
\cite[Theorem 2.6]{BCCD}.

\begin{conjecture}[{\cite[Conjecture 1.7]{HMYZ2023}}]\label{conjecture}
Let $f\in\CC[x_1,\ldots,x_s]$ be a weighted homogeneous polynomial defining a hypersurface with  isolated singularity. Assume that $\dw(f)\geq 2w_1\geq2w_2\geq\cdots\geq2w_s>0$. Then $\pT$ has no negative weighted derivations, for $n\geq2$.
\end{conjecture}

\begin{remark}
Recall that $\CC[x_1,\ldots,x_s]/\langle f, \mathcal{J}_1(f)\rangle$ coincides with the classical Tjurina algebra. The analogous statement of the previous conjecture for the Tjurina algebra is a conjecture formulated by Yau \cite{CXY1995}. To the best of our knowledge, Yau's conjecture is still open, although many cases have already been resolved \cite{CXY1995,XY1996,CYZ2019,CCYZ2020}.
\end{remark}

\section{Proof of Theorem \ref{theorem-main}}

\begin{proposition}\label{leastdeg}
Let $f\in\pol$ be a weighted homogeneous polynomial with respect to $w\in\NN_{\geq1}^s$. Assume that $\dw(f)\geq 2w_1\geq2w_2\geq\cdots\geq2w_s>0$. Then
\begin{equation*}
\dw(f)\leq \dw(g),
\end{equation*}
where $g$ denotes an arbitrary maximal minor of $\Jac_n(f)$, for $n\geq2$.
\end{proposition}
\begin{proof}
This can be checked directly for $s=n=2$ using the generators from Example \ref{j2}. Hence we can assume $s\geq3$ and $n\geq2$, or $s=2$ and $n\geq3$. In both cases, it is known that $\cJ_n(f)\subseteq\cJ_1(f)^3$ \cite[Proposition 2.20]{LeYasuda2025}. Let $d=\deg_w(f)$.

Denote $f_i=\frac{\partial f}{\partial x_i}$. Let $h=f_1^{\alpha_1}\cdots f_s^{\alpha_s}$ be a generator of $\cJ_1(f)^3$, where $\alpha_1+\cdots+\alpha_s=3$. Using the assumption $d\geq 2w_1\geq\cdots\geq2w_s>0$ we obtain
\begin{align*}
\deg_w(h)&=\alpha_1(d-w_1)+\cdots+\alpha_s(d-w_s)\\
&\geq \alpha_1(d-w_1)+\cdots+\alpha_s(d-w_1)\\
&=(\alpha_1+\cdots+\alpha_s)(d-w_1)\\
&=3(d-w_1)=d+2d-3w_1\\
&\geq d+4w_1-3w_1=d+w_1\geq d.
\end{align*}
Since $\cJ_1(f)^3$ is weighted homogeneous, the previous computation shows that any non-zero weighted element of $\cJ_1(f)^3$ has weighted degree greater than or equal to $d$. In particular, this is true for the generators of $\cJ_n(f)$.
\end{proof}

\begin{proof}[Proof of Theorem \ref{theorem-main}]
Since $f$ has an isolated singularity and is weighted homogeneous, the ideal generated by the partial derivatives $f_i=\frac{\partial f}{\partial x_i}$ is zero-dimensional. Then
$f_1,\ldots,f_s$ form a weighted homogeneous regular sequence of degrees $d-w_1,\ldots,d-w_s$, respectively. 

Let $\delta\in \Der_\CC (\pT)$ be a weighted homogeneous derivation of negative degree. Then $\delta$ has a weighted homogeneous lifting
$\tilde{\delta}\in \Der_\CC(\pol)$ of the same degree such that $\tilde{\delta}(\langle f,\cJ_n(f) \rangle)\subseteq \langle f,\cJ_n(f) \rangle$. Since $\tilde{\delta}$ has negative weighted degree, Proposition \ref{leastdeg} implies $\tilde{\delta}(f)=0.$
Let $r_1,\ldots,r_s\in \pol$ be such that
$\tilde{\delta}=r_1\frac{\partial}{\partial x_1}+\ldots+r_s\frac{\partial}{\partial x_s}$. Then, 
$$0=\tilde{\delta}(f)=r_1 f_1+\ldots+r_s f_s.$$
Since $f_1,\ldots,f_s$ are a regular sequence, we have that 
$$
\begin{pmatrix}
r_1\\
r_2\\
\vdots\\
r_s
\end{pmatrix}\in \langle f_ie_j-f_je_i\rangle_{i\neq j},
$$
where $e_i$ is the $i$-th vector in the canonical basis of $(\pol)^s$.
Hence, 
$$
\tilde{\delta}=r_1\frac{\partial}{\partial x_1}+\dots+r_s\frac{\partial}{\partial x_s}\in
\left\langle 
f_i \frac{\partial}{\partial x_j}-f_j\frac{\partial}{\partial x_i}\right\rangle_{i\neq j}.
$$
Then,
\begin{align*}
\dw(\delta)=\dw(\tilde{\delta})\geq &\min\left\{\dw \left( f_i \frac{\partial}{\partial x_j}-f_j\frac{\partial}{\partial x_i}\right)\right\}_{i\neq j}\\
=&\min\left\{ \dw(f)-w_i-w_j\right\}_{i\neq j}\geq 0.    
\end{align*}
This contradicts the fact that $\delta$ has negative degree.
\end{proof}

\begin{remark}
The assumption $\dw(f)\geq 2w_1\geq2w_2\geq\cdots\geq2w_s>0$ holds, up to a change of coordinates, whenever $f\in\langle x_1,\ldots,x_s\rangle^3$ \cite[Theorem 2.1]{Saito1971,MR3490075}.
\end{remark}

\begin{remark}
Let $f\in \CC[x_1,\ldots,x_s]$ be a weighted homogeneous polynomial with respect to $w\in\NN^s$ and degree $d$ that defines an isolated hypersurface singularity. Consider any set $S$ of weighted homogeneous polynomials satisfying $d\leq \deg_w(g)$, for all $g\in S$. Then, proceeding exactly as in the proof of Theorem \ref{theorem-main}, it follows that $\CC[x_1,\ldots,x_s]/\langle f, S\rangle$ does not admit negative weighted derivations.
\end{remark}


\end{document}